\documentclass{amsart}
\usepackage{amsfonts,amscd}

\newtheorem{theorem}{Theorem}[section]
\newtheorem{lemma}[theorem]{Lemma}
\newtheorem{corollary}[theorem]{Corollary}
\newtheorem{proposition}[theorem]{Proposition}
\theoremstyle{remark}

\theoremstyle{definition}

\numberwithin{equation}{section} \makeatother

\DeclareMathOperator{\Cdb}{{\mathbb C}}

\DeclareMathOperator{\Ndb}{{\mathbb N}}

\begin{document}

\title[Ordered involutive operator spaces]{Ordered involutive operator spaces}

\author[David P. Blecher]{David P. Blecher}
\address{Department of Mathematics, University of Houston, Houston, TX
77204-3008}
\email{dblecher@math.uh.edu}
\author[Kay Kirkpatrick]{Kay Kirkpatrick}
\address{Department of Mathematics, MIT, 77 Massachusetts Avenue,
    Cambridge, MA 02139-4307}
\email{kay@math.mit.edu}
 \author[Matthew Neal]{Matthew Neal}
\address{Department of Mathematics,
Denison University, Granville, OH 43023}
\email{nealm@denison.edu}
\author[Wend Werner]{Wend Werner}
\address{Mathematisches Institut, Einsteinstr. 62,
D-48149 M\"unster, Germany} 
\email{wwerner@math.uni-muenster.de} 
\thanks{To appear, {\em Positivity}}
\thanks{*Blecher was partially supported by grant DMS 0400731 from
the National Science Foundation.  Kirkpatrick was partially
supported by an NSF REU grant.  Neal was supported by Denison
University. Werner was supported by the SFB 487 Geometrische
Strukturen in der Mathematik, at the Westf\"alische
Wilhelms-Universit\"at, supported by the Deutsche
Forschungsgemeinschaft.}
 
\begin{abstract}
This is a companion to recent papers of the authors; here we
construct the `noncommutative
 Shilov boundary' of a (possibly nonunital) selfadjoint ordered
space of Hilbert space operators.  The morphisms in the universal property of the boundary
preserve order.
 As an application, we consider `maximal' and `minimal'
unitizations of such ordered operator spaces.
\end{abstract}

\maketitle

\section{Introduction}

  An operator space is a
closed linear space of bounded operators between Hilbert spaces, or,
equivalently, a subspace of a C*-algebra. Although
 ordered operator spaces containing the identity operator
  are well understood
 and important
 (these are known as operator systems \cite{CE,Pau}),
 {\em nonunital} ordered
operator spaces have barely been studied at all. This is despite the
fact that they occur very naturally; for example, consider the
linear span of three generic positive matrices in the $5 \times 5$
matrices $M_5$.   Indeed, the only theory addressing such spaces
which we are  aware of is contained in \cite{Sch,OSWU,WW,BW,BN} and
a series of papers by Karn (see e.g.\ \cite{K1,K2} and references
therein). In view of the importance of the notion of operator
positivity, we offer, in this companion paper to \cite{BW,BN}, some
results on this topic.  In particular, we
construct the `noncommutative
 Shilov boundary' of such a space $X$, and use this to construct a
`maximal unitization' of $X$. At the end of the paper we will
illustrate these two main concepts in the very simple
case of spaces of finite
matrices.

 Our main tool is the `noncommutative Shilov boundary'.  We
remark that this tool was used from the beginning in the case of
unital operator systems \cite{SOC}, and the commutative variant is
fundamental in the theory of classical function spaces. In 2001, as
a `Research Experience for Undergraduates' project, the second
author considered the possibility of a noncommutative Shilov
boundary for `selfadjoint operator spaces' \cite{TSBA}. It is well
known (see e.g.\ \cite{Ham,BLM}) that the noncommutative analogue of
the Shilov boundary of an operator space should be a {\em ternary
ring of operators}, or TRO for short.  A TRO is a closed subspace of
a C*-algebra which is closed under products of the form $x y^* z$.
If $X$ is selfadjoint then its enveloping TRO will also be
selfadjoint, that is, it is a `$*$-TRO'. In \cite{BW} the first and
last author developed the theory of $*$-TROs, and the ways in which
they can be ordered. If $X$ is an ordered operator space, then one
might hope that the
 `arrows' in the universal  diagram/property for the
 noncommutative Shilov boundary may be
chosen to be (completely) positive.  This was done in \cite{BN} in a
general setting.  In the present paper we give the technical details
of how to adapt these results to the case of selfadjoint operator
spaces. For example, the proof of the just mentioned universal
property is necessarily quite different in our setting here.  As an
application, in the last part of the paper,
 we discuss  `maximal' and `minimal'
unitizations of nonunital ordered operator spaces.

Any unexplained notation can be found in \cite{BW,BN,BLM}. A map $T
: X \to Y$ between vector spaces which have involutions is called
{\em $*$-linear} or {\em selfadjoint}, if it is linear and $i(x^*) =
i(x)^*$ for all $x \in X$. A {\em selfadjoint operator space} is an
operator space $X$ with an involution $* : X \to X$, such that there
exists a complete isometry $i : X \to B(H)$ which is $*$-linear.
Alternatively, a selfadjoint operator space is an operator space $X$
with an involution $* : X \to X$, such that
$$\Vert [x_{ji}^* ] \Vert = \Vert [x_{ij} ] \Vert , \qquad n \in
\Ndb, \; [x_{ij} ] \in M_n(X) . $$ (To see the nontrivial direction
of this equivalence, suppose that $X$ is a concrete (not necessarily
selfadjoint) subspace of $B(H)$, and that $\tau : X \to X$ is an
involution satisfying the last centered equation.  Then the
map $$X \; \longrightarrow \;  M_2(B(H)) : x \; \mapsto \; \left[ \begin{array}{ccl} 0 & x \\
\tau(x)^* & 0
\end{array} \right],$$
is a $*$-linear complete isometry of $X$ onto a selfadjoint subspace
of a C*-algebra.)

 By an {\em ordered operator
space}, we mean a selfadjoint operator space $X$ with a specified
cone ${\mathfrak c}_n \subset M_n(X)$ for each $n \in \Ndb$,  such
that there exists a completely isometric $*$-linear map $T : X \to
B(H)$ such that $T$ is completely positive in the sense that
$T_n({\mathfrak c}_n) \subset M_n(B(H))_+$, for all $n \in \Ndb$. We
warn the reader that this notation is nonstandard, in other papers
$T$ is usually also required  to be a {\em complete order
embedding}, that is $T^{-1}$ is also completely positive on
Ran$(T)$. An ordered operator space is often written as a pair
$(X,({\mathfrak c}_n))$. We sometimes write ${\mathfrak c}_n$ as
$M_n(X)_+$, and write $x \geq 0$ for $x$ in this cone.
 There is another important variant of the theory, where we have only
one cone ${\mathfrak c} \subset X$, and $T$ above is positive.
However since this variant works out almost identically, we will not
explicitly mention it.

We end this introduction with the following remark which we will
refer to later.  One initial motivation to study nonunital ordered
operator spaces comes from the fact that the dual of a C*-algebra
has a very nice positive cone (namely, the positive linear
functionals), and is an operator space \cite[Section 1.4]{BLM}.
This leads one to hope that it may be treated as an ordered operator
space. However the situation  here is certainly more difficult than
one might first think:

\begin{proposition} \label{one}   If $A'$ is the dual Banach space
of a
nontrivial C*-algebra $A$,
with its usual cone, then there exists no isometric positive map
from $A'$ into another C*-algebra.  Thus $A'$ is not an ordered operator space.
\end{proposition}

\begin{proof}  First suppose that $A = \ell^\infty_2$.
In this case this result was mentioned in \cite{BN}
as a consequence
 of some rather intricate facts.  However there is a simple
direct proof.   Suppose that
$u : \ell^1_2 \to B(H)$ is an isometric positive map.
 If $u(e_i) = T_i$, then $T_i$ are positive contractions.
Since $T_1 - T_2$ is selfadjoint, a well known
formula gives  $$\Vert T_1 - T_2 \Vert = \sup \{ |
\langle (T_1 - T_2) \zeta , \zeta \rangle | :
 \zeta \in {\rm Ball}(H) \} \leq 1 .$$
This contradicts  the fact that $2 = \Vert (1,-1) \Vert_{\ell^1_2}
 = \Vert T_1 - T_2 \Vert$.

We finish the proof by proving the more general fact that if $M$ is
any von Neumann algebra of dimension $> 1$,  then there exists no
positive isometry $T$ from the predual $M_*$ into another C*-algebra
$B$. For if there did exist such $T$, then $T'' : M' \to B''$ would
be a positive isometry into a C*-algebra. Let $p$ be a nontrivial
projection in $M$, and let $N = {\rm Span} \{ p,p^\perp \}$, the
injective 2 dimensional C*-algebra. By injectivity, there exists a
completely positive unital projection $P : M \to N$.  Then $P' :  N'
\to M'$ is a positive isometry. Thus $T'' \circ P'$ induces a
positive isometry from $\ell^1_2$ into a C*-algebra, contradicting
the previous paragraph.  \end{proof}

\section{The selfadjoint and the ordered noncommutative Shilov
boundary}

The natural  morphisms between TROs are the {\em ternary morphisms},
that is the maps satisfying $T(xy^*z) = T(x) T(y)^* T(z)$.
 As mentioned in the introduction of \cite{BW}, the basic
facts about TROs and their ternary morphisms have selfadjoint
variants valid for $*$-TROs and $*$-linear ternary morphisms.  We
will use the term `ternary $*$-morphism' for the latter. Similarly,
as we shall see next, there is a selfadjoint variant of the
noncommutative Shilov boundary of an operator space.  That is,
Hamana's theory of ternary envelopes (see \cite{Ham,BLM}) easily
restricts to the context of selfadjoint subspaces  of C*-algebras.

By a {\em ternary $*$-extension} of a selfadjoint operator space
$X$, we mean a pair $(Z,i)$ consisting of a $*$-TRO $Z$, and a
$*$-linear complete isometry $i : X \to Z$  such that `odd
polynomials' with variables taken from $i(X)$ are dense in $Z$.
Equivalently, there is no nontrivial subTRO of $Z$ containing
$i(X)$. We define a {\em ternary $*$-envelope} of $X$ to be any
ternary $*$-extension $(Y,i)$ with the universal property of the
next theorem.  This could also be called the `selfadjoint
noncommutative Shilov boundary'.

\begin{theorem} \label{7tru}
If $X$ is a selfadjoint operator space, then there exists a ternary
$*$-extension $(Y,j)$ of $X$ with the following universal property:
Given any ternary $*$-extension $(Z,i)$ of $X$ there exists a
(necessarily unique and surjective) selfadjoint ternary morphism
$\pi : Z \rightarrow Y$, such that $\pi \circ i = j$.
\end{theorem}

\begin{proof}  Let $(Z,i)$ be a ternary $*$-extension of $X$.
By the basic theory of the ternary envelope (see \cite{Ham} or
8.3.10 in \cite{BLM}), there exists a subbimodule  $N$
of $Z$ such that $Z/N$ is a copy of the
ternary envelope of $X$.
  For any operator space $X$, the adjoint space $X^* = \{ x^* : x \in
X \}$ is a well defined operator space, independent of the
representation of $X$ \cite[1.2.25]{BLM}. We have $(X/N)^* \cong
X^*/N^*$ completely isometrically, using \cite[1.2.15]{BLM}, for
example. Since the canonical map $j : X \to Z/N$ is completely
isometric, so is the canonical map $X = X^* \to (Z/N)^* \cong
Z^*/N^* = Z/N^*$.  However $N$ is the largest subbimodule $W$ of $Z$
such that the canonical map $X \to Z/W$ is completely isometric (see
\cite{Ham} or \cite[Theorem A11]{BShi}). Thus  $N^* \subset N$, so
that $N = N^*$. Hence the ternary envelope $Y = Z/N$ of $X$ is a
$*$-TRO, by observations in the introduction to \cite{BW}, and hence
is a ternary $*$-extension of $X$. If $(W,k)$ is any ternary
$*$-extension of $X$, then $W$ is a ternary extension of $X$, and by
the universal property of the envelope there exists a (necessarily
unique and surjective) ternary morphism $\pi : W \to Y$ with $\pi
\circ k = j$. It is easy algebra to check that $\pi$ is $*$-linear.
\end{proof}

It is now routine to extend the basic properties of ternary
envelopes (see e.g.\ \cite{Ham,BShi} or \cite[8.3.10 and
8.3.12]{BLM}) to ternary $*$-envelopes. We omit the details.
Indeed, part of the proof above shows, using routine arguments, that
any ternary $*$-envelope of $X$ is a ternary envelope of $X$.

Suppose next that $X$ is an ordered operator space, with  matrix
cones $({\mathfrak c}_n)$. We consider an ordered version of the
ternary envelope or noncommutative Shilov boundary.  That is, we
seek an ordered version of the universal property/diagram in Theorem
\ref{7tru}.  Except under extra hypotheses (for example, if the
positive cone of $X$ densely spans $X$), it is easy to show that the
embeddings $i : X \to Z$ occurring in the universal property/diagram
in the ordered case cannot be allowed in general to be arbitrary
complete order embeddings, or even arbitrary completely positive
complete isometries. We will usually need to limit the size of the
cone of $Z$. Fortunately there is an appropriate bound for this
cone, and this bound depends on the given cone of $X$.

More specifically, we assign a canonical cone ${\mathfrak d}$ to the
ternary $*$-envelope $({\mathcal T}(X),j)$, namely the intersection
of all natural cones (in the sense of \cite{BW}) containing
$(j({\mathfrak c}_n))$. To see that there exists at least one such
cone, note that if $i : X \to B$ is a completely positive complete
isometry into a C*-algebra, and if $W$ is the $*$-TRO generated by
$i(X)$, then by Theorem \ref{7tru}
 there is a
ternary $*$-morphism $\theta : W \to {\mathcal T}(X)$ with $\theta
\circ i = j$.  Thus ${\mathcal T}(X)$ is ternary $*$-isomorphic to a
quotient of $W$.  By \cite[Lemma 5.2]{BW}, this quotient of $W$ has
a natural  cone containing the image of $i({\mathfrak c})$.

As mentioned in \cite[Section 4]{BW}, the intersection of natural
cones (in the sense of \cite{BW}) is natural.  Hence ${\mathfrak d}$
is a natural cone on ${\mathcal T}(X)$ containing  $j({\mathfrak
c})$.
  We call ${\mathcal T}(X)$ equipped with this cone the
{\em ordered ternary envelope}, and write it as ${\mathcal T}^o(X)$.
This is the `ordered noncommutative Shilov boundary'.
  We remark that up to
the obvious notion of equivalence, the ordered ternary envelope is
well defined independently of any particular ternary $*$-envelope of
$X$.  This follows for example from the universal property below.

The following result is reminiscent of \cite[Theorem 5.3]{BN}, but
seems to require a completely different proof.  (Note that the proof
in the setting considered in \cite{BN} needs to use the notion of
`range tripotents'.  Since these may not be `central' in the sense
of \cite{BW}, they will not work for us here.

\begin{theorem}  \label{incot} Suppose that $(X,{\mathfrak c})$ is an
ordered operator space. Suppose that $\iota$ is a $*$-linear
completely positive complete isometry from $X$ into a $*$-TRO $Z$
such that $\iota(X)$ generates $Z$ as a $*$-TRO. Let ${\mathfrak e}$
denote the intersection of all natural cones of $Z$ (in the sense of
\cite{BW}) containing $\iota({\mathfrak c})$. Let ${\mathfrak d}$ be
the natural cone defining the ordered ternary $*$-envelope
${\mathcal T}^o(X)$. Then
 the canonical ternary $*$-morphism $\theta : Z \to
 {\mathcal T}^o(X)$ from Theorem {\rm \ref{7tru}}, takes
${\mathfrak e}$ onto ${\mathfrak d}$.
 \end{theorem}

\begin{proof} \ This requires basic results
and notation from \cite{BW}.
  Let $u$ be the tripotent in ${\mathcal T}(X)''$ associated (by the
correspondence from \cite[Theorem 4.16]{BW})
 with
the natural cone ${\mathfrak d}$ in ${\mathcal T}(X)$. We write
$J(u)$ for the `Pierce 2-space' of $u$ in ${\mathcal T}(X)''$,
namely the subset of elements satisfying
$y = u y u$. The inverse image under
the ternary $*$-morphism $\theta''$ of $J(u)$ is a weak*-closed
ternary $*$-ideal $I$ of $Z''$. Thus, $Z'' = I \oplus^{\infty} J$
for some ternary $*$-ideal $J$ (see e.g.\ \cite[Lemma 3.4 (1)]{BW}).
Now let $z$ be the tripotent associated with the natural cone
${\mathfrak e}$. Then $z = z_{1}+z_{2}$ for orthogonal tripotents
$z_{1}\in I$ and $z_{2}\in J$.
 Since $z$ is central and selfadjoint, it follows easily
that $z_{1}$ is central and selfadjoint.    Because $\theta''$ takes
open central tripotents to open central tripotents (by a variant of
\cite[Proposition 3.5]{BN}), and hence natural cones to natural
cones, $\theta''(z) \geq u$.  Since $z_{2}I=0$, it is also clear
that $\theta''(z_{2})$ is orthogonal to $u$.  Thus, $\theta''(z_{1})
\geq u$.  But $\theta''(z_{1}) \in J(u)$, and
so $\theta''(z_{1})= u
\theta''(z_{1})u = u$. It follows that $\theta''(J(z_{1})) \subset
J(u)$, where $J(z_{1})$ is the Pierce 2-space of $z$ in $Z''$.

    Recall that $\iota({\mathfrak c})$ lies in $J(z)_{+}$,
    the positive part of the Pierce 2-space of $z$.
      Thus, for
any $x \in \iota({\mathfrak c})$,  we have $x = z^2 x = z_{1}^2 x +
z_{2}^2 x$.  Since $x \in I$, we have $z_{2}^2 x$ = 0, so that $x =
z_{1}^2 x$. Hence, $\iota({\mathfrak c})$ lies in $J(z_{1})_{+}$,
which is a natural cone in $Z''$.
   It follows that $\iota({\mathfrak c})$ lies in the natural cone
${\mathfrak d}_{z_{1}}$ (see Lemma 4.15 (2) in \cite{BW}), which in
turn is contained in ${\mathfrak d}_z = {\mathfrak e}$. Hence
${\mathfrak e} = {\mathfrak d}_{z_{1}}$. Since $\theta''$ takes
squares of selfadjoint elements in $J(z_{1})$ to squares of
selfadjoints in $J(u)$, we have $\theta({\mathfrak e}) \subset
{\mathfrak d}$. Since $\theta$ takes natural cones to natural cones
(by a variant of \cite[Proposition 3.5]{BN}), by the definition of
${\mathfrak d}$ we must have
  $\theta({\mathfrak e}) = {\mathfrak d}$.  \end{proof}

The ordered noncommutative Shilov boundary is particularly nice in
the case that $X$ has a spanning cone: one may replace the cone
${\mathfrak e}$ defined in Theorem \ref{incot} by
 the entire positive cone of $Z$.

\begin{theorem}  \label{B}  Suppose that $X$ is an ordered operator space
with a densely spanning cone, and that $i : X \to B$ is a positive
complete isometry from $X$ into a C*-algebra. Then the TRO $A$
generated by $i(X)$ is a C*-subalgebra of $B$ (and equals the
C*-algebra generated by $i(X)$). Moreover, the ordered ternary
envelope of $X$ is a C*-algebra, which we write as $(C^*_e(X),j)$.
This C*-algebra $C^*_e(X)$ has the following universal property: for
$i : X \to A$ as above, there exists a (completely positive)
surjective $*$-homomorphism $\theta : A \to C^*_e(X)$ satifying
$\theta \circ i = j$.
\end{theorem}

\begin{proof}  This may be proved just as in
\cite[Corollary 5.4]{BN}.  We include a quick direct proof. Simple
algebra shows that $A^2$ and  $A \cap A^2$ are
 C*-subalgebras of $B$.
Also, if $x \in A \cap B_+$ then $x^2 \in (A^2)_+$, so that $x \in
A^2$ (since square roots in a C*-algebra remain in the algebra).
Thus $x \in A\cap A^2$. Hence $i(X_+) \subset A \cap A^2$, so that
$i(X) \subset A \cap A^2$.  This implies that $A \subset A \cap
A^2$, which forces $A$ to be a $C^*$-subalgebra of $B$.  The second
statement is now fairly obvious.

By the first lines of the proof of Theorem {\rm \ref{7tru}}, the
ternary envelope of $X$ may be chosen to be a quotient of $A$ by a
certain subspace.  This subspace is a two-sided ideal if $A$ is a
$C^*$-algebra, and so the quotient is a $C^*$-algebra which we write
as $C$ below. Moreover, the canonical map $j : X \to C^*_e(X)$ is
clearly also completely positive, being a composition of $i$ and the
quotient $*$-homomorphism.

Given another map $\kappa : X \to B'$ with the same properties as
$i$ above,  the TRO $D$ generated by $i(X)$ is a
C*-algebra,
 by the above. By the universal property of the ternary envelope,
 there
exists a ternary morphism $\pi : D \to C$ such that $\pi \circ
\kappa = j$. It is easy algebra to check that $\pi$ is $*$-linear.
Moreover, for example by the first part of the proof of
\cite[Corollary 8.3.5]{BLM}, the ternary morphism $\pi$ induces a
$*$-homomorphism $\theta : D^* D = A \to C^* C = C$ defined by
$\theta(d^* b) = \pi(d)^* \pi(b)$. If $x \in X_+$ then $\kappa(x)
\geq 0$, so that $\theta(\kappa(x)) \geq 0$. But also $j(x) =
\pi(\kappa(x)) \geq 0$.  Since $\theta(\kappa(x))^*
\theta(\kappa(x)) = \theta(\kappa(x)^* \kappa(x)) = \pi(\kappa(x))^*
\pi(\kappa(x))$, it follows from the uniqueness of square roots that
$\pi(\kappa(x)) = \theta(\kappa(x))$. Since $\kappa(X_+)$ generates
$D$ as a TRO, this forces $\pi = \theta$.

Finally, we show that there is only one natural cone
 in $C^*_e(X)$  containing  $X_+$.  Any such cone
may be viewed as the cone of a $*$-subTRO $W$ of a $C^*$-algebra,
for which there is a ternary $*$-isomorphism $\rho : C \to W$ with
$\kappa = \rho_{\vert X}$ positive. Applying the previous paragraph
to $\kappa$, easily shows that $W$ is a $C^*$-algebra and
$\rho^{-1}$ is a $*$-isomorphism, and hence an order isomorphism.
\end{proof}

Indeed most of the other results  in \cite[Section 5]{BN} carry over
 to our setting.

We call $C^*_e(X)$ in Theorem \ref{B}, the {\em C*-envelope} of $X$.
If $X$ is a unital operator system then this coincides with the
usual C*-envelope of Arveson and Hamana (see \cite{SOC,Ham1} or
\cite[Section 4.3]{BLM}).

 \medskip

 {\em Remarks.} 1) \ If $X$ is as in Theorem \ref{B}, then the
 universal
property there easily implies that there is a largest ordered
operator space cone on $X$ containing the given cone ${\mathfrak c}
= X_+$, namely $X \cap C^*_e(X)_+$. Thus the embedding of $X$ in
$C^*_e(X)$ is a complete order embedding iff there is no strictly
larger ordered operator space cone on $X$ containing ${\mathfrak
c}$.

2)  From the universal property it is easy to see that $C^*_e(X)$ is
also an ordered ternary envelope of $X$ in the sense of \cite{BN}.

\medskip

We say that $X$ in Theorem \ref{B} is {\em e.n.v.} if $C^*_e(X)$ is
unital. This property only depends on the operator space structure
of $X$ (not the order) \cite{BShi}. In \cite[Section 4]{BShi}, it is
shown that if $X$ is a `minimal operator space' (that is, $X$ is
completely isometric to a subspace of a $C(K)$ space), then $X$ is
e.n.v. iff $0$ is not a weak* limit of extreme points of ${\rm
Ball}(X')$. Typically, spaces will not be e.n.v..

 For the following result, we recall that the classical
Shilov boundary is usually shown to exist for function spaces which
contain constant functions. A space not containing constant
functions may not have a Shilov boundary in the usual sense.  We
prove that this boundary does exist in a case we have not seen
discussed in the classical literature:

\begin{corollary} \label{apsh}  Suppose that $X$ is a
closed selfadjoint subspace of $C_0(K)$, for a locally compact
Hausdorff $K$, and suppose that the cone $X \cap C_0(K)_+$ densely
spans $X$.
 Then the Shilov boundary of
$X$ exists as a topological space in the following sense: there
exists a locally compact topological space $\partial X$, and a
positive linear isometry $j : X \to C_0(\partial X)$ such that
$j(X)$ separates points of $\partial X$ and does not vanish
identically at any point, and such that for any other locally
compact topological space $\Omega$ and positive linear isometry $i :
X \to C_0(\Omega)$ such that $i(X)$ strongly separates points and
does not vanish identically at any point,
 there
exists a homeomorphism $\tau$ from $\partial X$ onto a subset of
$\Omega$, such that $i(x) \circ \tau = j(x)$ for all $x \in X$.
Also, $\partial X$ is compact iff  $X$ is e.n.v..
\end{corollary}

\begin{proof}     The TRO inside $C_0(K)$ generated
by $X$ is a C*-algebra $A$ by Theorem \ref{B}, and it is
commutative. By the universal property in Theorem \ref{B}, there is
a $*$-homomorphism from $A$ onto  $C^*_e(X)$.  Thus $C^*_e(X)$ is a
 commutative
C*-algebra, so that $C^*_e(X) = C_0(\partial X)$ for a locally
compact topological space $\partial X$, which will
be compact if $X$ is e.n.v..  Clearly the copy of $X$
separates points of $\partial X$ and does not vanish identically at
any point.   Given any $(\Omega,i)$ as stated,
 then
$(C_0(\Omega),i)$ is a ternary $*$-extension of $X$, so that by
Theorem \ref{B} there is a surjective $*$-homomorphism $\pi :
C_0(\Omega) \to C_0(\partial X)$ as in that Corollary, with $\pi
\circ i = j$. By the well known dualities between categories of
topological spaces and algebras of continuous functions, $\pi$
induces a homeomorphic embedding $\tau :
\partial X \to \Omega$ such that $i(x) \circ \tau = j(x)$.
\end{proof}

The C*-envelope in Theorem \ref{B}, and the Shilov boundary in the
last result, will have (ordered versions of) almost all of the usual
familiar properties of the noncommutative and classical Shilov
boundaries (see e.g.\ \cite{BLM,Ham} for some of these).   Thus it
should be a useful tool.  For example, it gives one a good handle on
the completely positive surjective complete isometries $T : X \to Y$
between ordered operator spaces which are densely spanned by their
positive cones. Indeed, such maps extend uniquely to
$*$-isomorphisms between the $C^*$-envelopes, where in certain cases
(such as the finite dimensional case considered at the end of the
paper) they can be classified. To see this, note that the universal
property in the Theorem \ref{B}, applied to $j' \circ T$, gives a
$*$-homomorphism $\pi : C^*_e(X) \to C^*_e(Y)$ such that $\pi \circ
j = j' \circ T$, where $j' : Y \to C^*_e(Y)$ is the canonical map.
On the other hand, by basic properties of the noncommutative Shilov
boundary of any operator space, there is a ternary isomorphism
$\theta : C^*_e(X) \to C^*_e(Y)$ such that $\pi \circ j = j' \circ
T$.   So $\pi = \theta$ on $j(X)$, and hence on all of $C^*_e(X)$.
So $\pi$ is bijective, and so is a $*$-isomorphism, and the
uniqueness is obvious.

\section{Unitizations}  A natural way to
attempt to understand nonunital ordered operator spaces is to
`unitize' them; that is to embed them as a codimension one subspace
of a unital operator system (in the usual sense of e.g.\ \cite{CE}).
This was first done in \cite{OSWU}. That paper assigns, to any
`matrix ordered operator space' $X$, a unitization $X^+$, which is
an operator system.  Note $X^+$ is spanned by $X$ and $1$, and the
embedding  $i : X \to X^+$ is a completely contractive complete
order embedding.   One cannot hope that $i$ be isometric too, in
general.   To see this, suppose that $i$ was isometric in the  case
that $X = A'$ is the dual of a nontrivial C*-algebra $A$, with its
usual cone (which is a `matrix ordered operator space'). Since $X^+$
is an operator system, and hence may be viewed as a subspace of
$B(H)$ containing $I_H$, we  have contradicted Proposition
\ref{one}.  (We remark in passing that the same argument shows that
the main result in \cite{K1} is not correct as stated; it has been
corrected in \cite{K2}.)  It is easy to see from e.g.\
\cite[Corollary 4.11]{OSWU}, that $i$ is a complete isometry iff $X$
is a completely isometrically complete order embeddable in a
C*-algebra.  We mention, in passing, that such spaces were
characterized in \cite{OSWU}, as the spaces whose `extended
numerical radius' norms coincide with the given matrix norms.

It is shown in \cite[Lemma 4.9 (c)]{OSWU} that $X^+$ has the
following universal property:  for any completely contractive
completely positive map $T$ from $X$ into a (unital) operator system
$Y$, the extension $x + \lambda 1 \mapsto T(x) + \lambda 1$, from
$X^+$ into $Y$ is completely positive.  Since the latter map is
unital, it is also completely contractive. From this, it follows
that $X^+$  possesses the smallest cone a unitization of $X$ may
have. To see this, take  $Y$ in the universal property above to be
any other unitization of $X$.

In contrast, it may well be useful sometimes to have a unitization
with a `biggest' positive cone.  We will show that such a `big
unitization' exists if $X_+$ densely spans $X$. In fact, after
playing with examples, one sees that sometimes it is easier to
describe this `big unitization' explicitly than $X^+$ (see e.g.\ the
example at the end of the paper, or Corollary \ref{notenv}).
 For ordered operator spaces, we will define a unitization
 $X^1$ via the
results in Section 2. Namely, if $X$ is an ordered operator space,
let $W = {\mathcal T}^o(X)$ be its ordered ternary envelope. Since
the positive cone of this is `natural' in the sense of \cite{BW}, by
the results of that paper we may consider $W$ as a $*$-TRO inside a
C*-algebra, with the inherited cone. Then $B = W + W^2$ is a
C*-algebra (see \cite{BW}), and we set $X^1$ be the span of the
image of  $X$ and the identity of the C*-algebra unitization $B^1$
of $B$. There is a choice to make here if $B$ is already unital: in
this paper we take $B^1 = B$ in this case.  Note that this choice
has the consequence that $X^1 = X$ if $X$ is already a unital
operator system.

\begin{theorem}  \label{upuni}
Suppose that $X$ is an ordered operator space with a densely
spanning cone (resp.\ $X$ is an ordered operator space), that $H$ is
a Hilbert space, and that $i : X \to B(H)$ is a $*$-linear complete
isometry which is completely positive (resp.\ completely positive
and  such that there is no smaller natural cone than $A \cap
B(H)_{+}$ in the TRO $A$ generated by $i(X)$ containing the image of
the cone of $X$). Then there is a completely positive unital map
from $i(X) + \Cdb I_H   \to   X^1$ extending the canonical map $i(X)
\to X$.
\end{theorem}

\begin{proof}   Let
$\theta$ be the canonical ternary $*$-morphism from the TRO $A$
generated by $i(X)$ to ${\mathcal T}^o(X)$ (see Theorem  \ref{incot}
and Theorem  \ref{B}). If $X$  has a densely spanning cone, then by
Theorem \ref{B}, $A$ is a C*-algebra,  and  $\theta$ is a
$*$-homomorphism, and hence extends to a $*$-homomorphism $\pi$ from
a unitization  of $A$, which we may take to be the span of $A$ and
$I_H$, into a unitization of the C*-algebra  ${\mathcal T}^o(X)$. We
then restrict $\pi$ to the span of $i(X)$ and $I_H$.

In the `respectively case', $\theta$ is positive by Theorem
\ref{incot}, and hence is positive as a map into the C*-algebra $B =
W + W^2$ discussed above.    By \cite[Corollary 4.3]{BW} we can
extend $\theta$ to a completely positive unital $*$-homomorphism
from a unitization  of $Z + Z^2$ into the unitization of $W + W^2$.
Here $Z$ is the TRO generated by $i(X)$. Finally, restrict to the
span of $i(X)$ and $I$ as before.
 \end{proof}

\begin{corollary} \label{ap}  If $X$ is an ordered
operator space with a densely spanning cone, then $C^*_e(X^1) =
C^*_e(X)^1$. \end{corollary}

\begin{proof}
 We assume that $C^*_e(X)$ is nonunital (the other case is
similar but easier).  By a basic property of the C*-envelope (see
e.g.\ \cite{Ham1,BLM}), it suffices to show that $J = 0$ if $J$ is
an ideal in $C^*_e(X)^1$ such that the canonical map $d : X^1 \to
C^*_e(X)^1/J$ is a complete isometry. Let $I = J \cap C^*_e(X)$,
then the canonical map $C^*_e(X) \to C^*_e(X)^1/J$ factoring through
$C^*_e(X)^1$, has kernel $I$. Hence $d$ restricted to $X$ may be
viewed as a complete isometry into $C^*_e(X)/I$. By a basic property
of the ternary envelope (see \cite{Ham} or \cite[8.3.10]{BLM}), we
have $I = 0$. Since $C^*_e(X)$ is an essential ideal in
$C^*_e(X)^1$, we deduce that $J = 0$. Thus $C^*_e(X^1) =
C^*_e(X)^1$.
\end{proof}

In this paragraph, let $X$ be an `$L^\infty$-matricial Riesz normed
space' in the sense of \cite{Sch,K1,K2}.   This is a subclass of the
`matrix ordered operator spaces' from \cite{OSWU} (to see that $X$
satisfies the first condition in Definition 3.3 of \cite{OSWU}, use
the formula in \cite[Definition 2.3]{K1} and pre- and post-multiply
the matrix there by the permutation matrix $U$ that switches around
the four corners.  The second condition is explicitly noted in
\cite{K1}).
 In \cite{K1,K2} a unitization for
 such spaces $X$ is introduced, which we shall write as
 $\tilde{X}$.    We show next that this unitization coincides with
the one in \cite{OSWU}. Thus one can apply results in \cite{K1,K2}
to the setting of \cite{OSWU}, and vice versa. For example, it
follows from \cite{K2} that there is a nice `internal'
 description
the cone on $X^+$ in the case that $X$ is `$L^\infty$-matricial
Riesz normed': namely, if $v \in M_n(X)$ and $A \in M_n$ are
selfadjoint, then $(v,A) = [v_{ij} + a_{ij} 1]$  is in the
positive
cone of $M_n(X^+)$ iff $A \geq 0$ and for any  $\epsilon
>0$ there is an $u\in M_n(X)_+$ with $\|u\|<1$ so that
$v+(A+ \epsilon)^{1/2}u(A+\epsilon)^{1/2}\geq 0$ (see \cite{K1}). 
 The first author
thanks A. K. Karn for sending him a version of \cite{K2}, which
contains another variant of this formula that is not hard to see is
equivalent to the one above (to see the nontrivial direction,
note that if $\theta$ and $u$ are as in the definition of the cone
in \cite{K2}, put
$$
\hat{u}= (A + \epsilon
1)^{-1/2}(A+\theta)^{1/2}u(A+\theta)^{1/2}(A+\epsilon
 1)^{-1/2}
$$
Since $A+\theta\leq A + \epsilon 1$ we have
$$
\|(A + \epsilon 1)^{-1/2}(A+\theta)^{1/2}\|^2=\|(A+ \epsilon
1)^{-1/2}(A+\theta)(A + \epsilon 1)^{-1/2}\| \leq 1 .
$$
Thus $\|\hat{u}\|<1$, and  $v + (A + \epsilon 1)^{1/2}
\hat{u} (A + \epsilon 1)^{1/2} = v+
(A+\theta)^{1/2}u(A+\theta)^{1/2} \geq 0$.)

\begin{proposition}
Whenever $X$ is as in the last paragraph, then $X^+= \tilde{X}$,
completely positively and completely isometrically via the identity
map.
\end{proposition}

\begin{proof}  Consider the inclusion $X \to \tilde{X}$.
By the property of $X^+$ mentioned at the start of this section,
this extends to a completely positive unital map $X^+ \to \tilde{X}$.
It therefore remains to show that every positive element $(v,A)$ in
$M_n(\tilde{X})$ is positive in $M_n(X^+)$. Let $\epsilon
>0$ be given, and let $u$ be as above, with
$v+(A+ \epsilon)^{1/2}u(A+\epsilon)^{1/2}\geq 0$. Then, for any
positive contractive $\phi\in M_n(X)'$,
$$
\phi\left((A+\epsilon)^{-1/2}v(A+\epsilon)^{-1/2}\right)\geq
\phi(-u)\geq -1.
$$
This is precisely the definition from \cite{OSWU} of the matrix
cones of $X^+$.  That is, by definition, $(v,A)\geq 0$ in
$M_n(X^+)$.
\end{proof}

As we see next, there seem to be interesting classes of ordered
operator spaces $X$ for which $X^+ = X^1$. Note that the latter
condition implies, first, that $X  \subset X^+$ completely
isometrically.  A unitization $\tilde{X}$ of $X$ will be called a
{\em completely isometric unitization} if the canonical map $X \to
\tilde{X}$ is completely isometric. Second, if $X^+ = X^1$ and
 if
$X$ is densely spanned by $X_+$, then all completely isometric
unitizations of $X$ coincide. This is because $X^+$ is the smallest
unitization, and Theorem \ref{upuni} implies that $X^1$ has the
biggest positive cone of any completely isometric unitization.

For example, we now study ordered operator spaces $X$ with the
property that $X''$, with its induced positive cones (the canonical
ones from e.g.\ \cite{Sch}, which are easy to see are just the weak*
closure of $M_n(X)_+$), is a unital operator system (that is, it has
an order unit making its cones Archimidean \cite{CE}).  We sometimes
write $1$ for the identity $1_{X''}$ of $X''$.   We will show that,
under reasonable conditions, $X^+ = X^1 = {\rm Span}\{ X, 1_{X''}
\}$, the span in $X''$.  In \cite{K1} a condition on `approximate
matrix order unit spaces' is given which characterizes when $\tilde{X}
 = {\rm Span}\{ X, 1_{X''} \}$.
  We give a variant of this result, and include a short proof.

\begin{lemma} \label{note}  {\rm (Cf.\  \cite[Theorem 5.3]{K1})}
 Suppose that $X$ is
an ordered operator space, and that $X''$ is a unital operator
system as above.  Then $X^+ = {\rm Span}\{X,1_{X''} \}$ iff there is
no $v \in X$ dominating $1_{X''}$, and iff the distance $d(X,1)$ of
$1_{X''}$ from the copy of $X$ in $X''$ is $1$.
\end{lemma}

\begin{proof}   In this case,  by the remark at the end of
the first paragraph of this section, $X \subset X^+$ completely
isometrically. By the universal property of this construction,
stated at the start of this section, the canonical map $i : X^+ \to
(X'')^+$ is unital and completely positive. On the other hand, by
definition of this unitization (see \cite[Definition 4.7]{OSWU}), if
$(v,A) \in M_n(X \oplus \Cdb)$ is in $M_n((X'')^+)$ then it is easy
to see that it is in $M_n(X^+)$. Consequently $i$ is a complete
order embedding.   Thus $X^+$ is easy to describe explicitly in
terms of the simple unitization of the already unital operator
system $X''$. Namely, the positive cone of $M_n(X^+)$ consists of
pairs $(v,A)$, with $v \in M_n(X)$ and $A \in (M_n)_+$, such that $v
+ A 1_{X''}$ is positive in $M_n(X'')$. It follows now from
\cite[Lemma 4.9(b)]{OSWU} that we can identify $X^+$ with the span
of $X$ and $1_{X''}$ in $X''$,  iff whenever $v + A 1_{X''}$ is
positive in $M_n(X'')$ then $A \geq 0$. As noted in \cite{K1}, we
can take $n = 1$ in the latter condition.
 This, in turn, is clearly equivalent to there being no $v \in X$
 dominating $1_{X''}$.

Suppose that $d(X,1) < 1$.  Then there exists $x \in X$ with $\Vert
x - 1 \Vert < 1$. By the usual trick we may assume $x = x^*$. It
follows by spectral theory that there is a $\lambda > 0$ with $x
\geq \lambda 1$, so that $x/\lambda \geq 1$. Conversely, if $d(X,1)
= 1$, then by the Hahn-Banach theorem there exists a state $\varphi$
on $X''$ which annihilates the copy of $X$.  If $v \in X$ with $v
\geq 1_{X''}$, then $0 = \varphi(v) \geq 1$, a contradiction.
\end{proof}

\begin{theorem} \label{notenv}  If $X$ is
an ordered operator space with $X''$  a unital operator system as
above, then $X^1 = {\rm Span}\{X,1_{X''} \}$. In this case, $d(X,1)
= 1$ iff all unitizations of $X$ coincide.  In particular, the
latter is true if $X$ is not e.n.v. (which is generally the case).
This holds, for example, for any nonunital C*-algebra.
\end{theorem}

\begin{proof} Note that  $X''$ is spanned by its positive
elements, and it follows that $X$ is densely spanned by $X_+$. Hence
the ordered ternary envelope of $X$ is a C*-algebra $C^*_e(X)$. By
\cite[Lemma 5.3]{BHN} in conjunction with Theorem \ref{B}, we may
view $C^*_e(X)$ as a C*-subalgebra $B$ of $C^*_e(X'')$.  So if $X$
is not e.n.v., then we can identify $C^*_e(X)^1$ with the span of
$B$ and $1_{X''}$. Thus in this case $X^1$ is the span of $X$ and
$1_{X''}$ in $X''$, up to unital complete order isomorphism (and
complete isometry). Similarly, if $X$ is e.n.v., with $1_{X''} \in
C^*_e(X)$, then $X^1 = {\rm Span}\{ X, 1_{X''} \}$ by definition.
Suppose that $X$ is e.n.v., but the identity $p$ of $C^*_e(X)$ is
not $1_{X''}$. If $x \in X$ then $\Vert 1_{X''} - x \Vert = \Vert
1_{X''} - p x p \Vert \geq \Vert p^\perp \Vert = 1$.  Thus $d(X,1) =
1$.   As in Lemma \ref{note}, there exists a state $\varphi$ on
$X''$ annihilating the copy of $X$.  Let $Y = {\rm Span}\{ X,p\}$
and $Z = {\rm Span}\{ X, 1_{X''} \}$, spans taken in $C^*_e(X)''$.
The map $\Phi :  z \mapsto p z p$ on $C^*_e(X'')$ restricts to a
unital completely positive map $Z \to Y$.  The canonical
$*$-isomorphism $C^*_e(X) + \Cdb 1_{X''} \cong C^*_e(X)
\oplus^\infty \Cdb$ taking $z = x + \lambda 1_{X''}$ to $(x +
\lambda p , \lambda)$ for $x \in C^*_e(X), \lambda \in \Cdb$,
restricts on $Z$ to the map $z  \mapsto (\Phi(z), \varphi(z))$.
Since $\varphi$ is completely contractive, this forces $\Phi$ to be
completely isometric. Thus $\Phi$ is a completely isometric complete
order isomorphism, so that $X^1 \cong {\rm Span}\{ X, 1_{X''} \}$.

If $X$ is not e.n.v. then the condition in Lemma \ref{note} holds.
This is clear if $X$ is a C*-algebra, and the general case reduces
to this one as follows. In this case, we can identify $v + \lambda
1_{X''}$, as in the last paragraph, with $v + \lambda 1$ in
$C^*_e(A)^1$.  The condition in Lemma \ref{note} now follows by the
C*-algebra case.

The rest follows easily from the above proof, together with remarks
and facts above the theorem.
\end{proof}

{\em Remark.} If $X''$ is a unital operator system, as in the
discussion above, then it is easy to argue that $X$ is `maximally
ordered': i.e.\
there is no larger ordered operator space cone on $X$.  We
leave the details to the interested reader.

\medskip

{\em Example.}  We illustrate the two main concepts of our paper
in the very simple case of a selfadjoint subspace $X$
of matrices in $M_n$, by recalling for the 
general reader the usual
 recipe to compute the ordered noncommutative Shilov boundary;
and this gives the unitization $X^1$ explicitly.  
In practice, for example
given a concrete subspace of $M_5$ say, it is usually easy to
implement this recipe.   For simplicity in the discussion that
follows, suppose that $X$ is spanned by $X \cap (M_n)_+$, which is
equivalent to having a set of fewer than $n^2$ linearly independent
positive matrices which span $X$ (the general case is only
fractionally more complicated). Then by Theorem \ref{B}, the TRO
generated in $M_n$ is also the C*-algebra $B$ generated by $X$ in
$M_n$. However every finite dimensional C*-algebra `is just' a
direct sum of full matrix algebras
 $M_k$, for $k \leq n$.  Thus, we may rewrite $X \subset B = B_1 \oplus
  B_2 \oplus \cdots \oplus
B_n$, where each $B_i = M_k$, for some $k \leq n$.  We call the
$B_i$ `blocks', and let $p_i$ be the minimal central projection in
$B$ which `is' just the identity of block $B_i$ (direct summed with
several zeros). Some of these blocks may be unnecessary in the
computation of norms of matrices with entries in $X$, and we call
such a block `loose'.  More specifically, if $\Vert x \Vert = \Vert
x (1-p_i) \Vert$ for all $x \in X$, and more generally for all $x
\in M_n(X)$, then block $B_i$ is loose.   If we find a loose block
we remove it.  For example, if block $B_1$ is loose, then  we
replace $B$ by $B' = B_2 \oplus \cdots \oplus B_n$, and replace $X$
by the subset of $B'$ corresponding to $X (1-p_1)$. We then repeat
the procedure. After less than $n$ steps, there will be no more
loose blocks, and we will have arrived at the ternary $*$-envelope,
which in this case is $C^*_e(X)$, and by construction
it contains a completely
isometric copy of $X$.  The span of the identity of the
latter $C^*$-algebra and the copy of $X$ inside this $C^*$-algebra,
is $X^1$.

\medskip

{\bf Acknowledgement.}  We thank A. K. Karn for a correction
of a remark in an earlier version of this paper.

\end{document}